\documentclass[11pt]{amsart}
\usepackage{amsmath,amssymb,mathrsfs,color}

\let\cal=\mathcal
\def\N{{\mathbb N}}

\def\R{{\mathbb R}}

\newtheorem{thm}{Theorem}[section]
\newtheorem{cor}[thm]{Corollary}
\newtheorem{lem}[thm]{Lemma}
\newtheorem{prop}[thm]{Proposition}

\theoremstyle{definition}
\newtheorem{de}[thm]{Definition}
\theoremstyle{remark}
\newtheorem{rem}[thm]{Remark}

\numberwithin{equation}{section}

\newcommand{\rmd}{{\rm d}}
\newcommand{\rme}{{\rm e}}
\allowdisplaybreaks

\begin{document}

\title[Favard separation method for a.p. SDEs]{Favard separation method for almost periodic stochastic differential equations}

\author{Zhenxin Liu}
\address{Z. Liu: School of Mathematics,
Jilin University, Changchun 130012, P. R. China and School of Mathematical Sciences,
Dalian University of Technology, Dalian 116024, P. R. China}
\email{zxliu@jlu.edu.cn; zxliu@dlut.edu.cn}

\author{Wenhe Wang}
\address{W. Wang: School of Mathematics,
Jilin University, Changchun 130012, P. R. China}
\email{whwang13@mails.jlu.edu.cn}

\thanks{This work is partially supported by NSFC Grants 11271151, 11522104, and the startup and Xinghai
Youqing funds from Dalian University of Technology.}

\date{October 20, 2015}



\begin{abstract}
Favard separation method is an important means to study almost periodic solutions to
linear differential equations; later, Amerio applied Favard's idea to nonlinear differential equations.
In this paper, by appropriate choosing separation and almost periodicity in distribution sense, we obtain
the Favard and Amerio type theorems for stochastic differential equations.
\end{abstract}

\maketitle

\section{Introduction}

The theory of almost periodic functions was founded by Bohr in 1924--1926 \cite{B1,B2,B3}, and many significant contributions were made to the subject
in the immediate decade following Bohr's work; see, for example, Bochner \cite{Boch27,Boch33}, von Neumann \cite{N}, van Kampen \cite{K}.
In the early stage of the theory, much attention was paid to the Fourier series theory of almost periodicity. Later it was observed that many differential equations
arising from physics admit almost periodic solutions, then almost periodic phenomenon
was extensively studied in differential equations, following Favard's pioneering work \cite{F,F33}; we refer the reader to the books,
e.g. Amerio and Prouse \cite{AP}, Fink \cite{Fink74}, Levitan and Zhikov \cite{LZ},
Yoshizawa \cite{Y75} etc, for the survey.

We know that the white noise perturbations have the effect of mixing and averaging, so what will happen when the almost periodic equation in consideration is perturbed
by white noise? In this situation, the almost periodic phenomenon was studied in stochastic differential equations. To the best of our knowledge,
only the fixed point method was used so far to investigate the existence of almost periodic solutions by assuming that
the linear part of the equation admits the exponential dichotomy; see Halanay \cite{Hal}, Morozan and Tudor \cite{MT}, 
Da Prato and Tudor \cite{DT}, and Arnold and Tudor \cite{AT}, among others.


In this paper, we aim to adopt the Favard separation method to study almost periodic solutions for stochastic differential equations.
The separation method goes back to Favard \cite{F} for linear equations.
Consider\footnote{To compare the results for periodic solutions, we
reverse somewhat the history of the subject.} the linear equation on $\R^d$
\begin{equation}\label{LE}
\dot x=A(t)x + f(t).
\end{equation}
If $A$ and $f$ are periodic with common period, the classical Massera criterion \cite{M} states that \eqref{LE} admits a periodic solution
with the same period if and only if it admits a bounded solution. When $A$ and $f$ are almost periodic, the situation is
more complicated. When $A$ is a constant matrix and $f$ is almost periodic in \eqref{LE}, Bohr and Neugebauer \cite{BN} proved that
a solution of \eqref{LE} is almost periodic if and only if it is bounded. But in the general case,
the existence of bounded solutions of \eqref{LE} does not imply the existence of almost periodic ones;
see counterexamples given by Zhikov and Levitan \cite{ZL}, Johnson \cite{Joh},
and the more recent work of Ortega and Tarallo \cite{OT} which unifies the situations of \cite{ZL,Joh}. Assuming the existence of bounded solutions,
Favard \cite{F} proved that \eqref{LE} admits an almost periodic solution if the so-called
{\it Favard separation condition} holds, which means that, for any $B\in H(A)$, each nontrivial bounded solution $x(t)$ of the equation
\begin{equation*}\label{hullb}
\dot x=B(t)x,
\end{equation*}
satisfies
\[
\inf_{t\in\mathbb R} |x(t)| >0.
\]
Here the hull $H(A)$ of $A$ is defined as follows
\[
H(A)= {\rm cl} \{A_\tau: \tau\in\R\}
\]
with $A_\tau(\cdot)=A(\tau+\cdot)$ and the closure being taken under the uniform topology. The Favard separation condition is optimal in some sense
since all the counterexamples we know so far (e.g. the works mentioned above) fail to obey it.  The Favard separation condition was extensively studied in
the literature in various situations. In particular, Amerio \cite{A} applied Favard's idea to nonlinear differential equations to study
almost periodic solutions. Later, Seifert \cite{S65} proposed a kind of separation, which is equivalent
to almost periodicity, to study the almost periodic solutions
of nonlinear equations. Fink \cite{Fink72} generalized separation conditions of \cite{A,S65} to semi-separation ones.

What we mainly concern in the present paper is the existence of almost periodic in distribution solutions to
stochastically perturbed differential equations under
the Favard or Amerio type separation condition. For instance, when \eqref{LE} is perturbed by small white noise:
\[
\rmd X = (A(t)X+f(t)) \rmd t + \epsilon \rmd W,
\]
does it admit almost periodic solutions in some sense if the unperturbed equation admits bounded solutions and satisfies the Favard separation condition?
To this interesting question, the answer is positive. Actually, we can obtain more general result than this; see the following
Favard type theorem.

\medskip
\noindent {\bf Theorem A.} {\em Consider the It\^o stochastic differential equation on $\R^d$
\begin{equation}\label{SLE}
\rmd X = (A(t)X+f(t)) \rmd t + \sum_{i=1}^m(B_i(t)X + g_i(t)) \rmd W_i,
\end{equation}
where $A, B_1,\ldots, B_m$ are $(d\times d)$-matrix-valued and $f,g_1,\ldots,g_m$ are $\R^d$-valued;
all of these functions are almost periodic; $W=(W_1,\ldots,W_m)$ is a standard $m$-dimensional Brownian motion.
Assume that \eqref{SLE} admits an $L^2$-bounded solution $X$, i.e. $\sup_{t\in\R}E|X(t)|^2<\infty$, and that the
Favard separation condition holds for \eqref{SLE}. Then \eqref{SLE} admits an almost periodic in distribution solution.}
\medskip

For nonlinear stochastic differential equations, we have the following Amerio type theorem.

\medskip
\noindent {\bf Theorem B.} {\em Consider the It\^o stochastic differential equation on $\R^d$
\begin{equation}\label{SNE}
\rmd X = f(t,X) \rmd t + g(t,X) \rmd W,
\end{equation}
where $f(t,x)$ is an $\R^d$-valued uniformly almost periodic function, $g(t,x)$ is a $(d\times m)$-matrix-valued
uniformly almost periodic function, and $W$ is a standard $m$-dimensional Brownian motion. Assume that
$f$ and $g$ are globally Lipschitz in $x$ with Lipschitz constants independent of $t$, and that
the Amerio semi-separation condition holds for \eqref{SNE} in $\cal D_r$ for some $r>0$. Then all the $L^2$-bounded solutions
of \eqref{SNE}, with $\sup_{t\in\R}E|X(t)|^2\le r^2$, are
almost periodic in distribution.}
\medskip

For the Favard separation condition for \eqref{SLE} and the Amerio semi-separation condition for \eqref{SNE} as well as
the meaning of $\cal D_r$, we refer the reader to Section 2 for details.

Besides the above Theorems A and B, we also obtain a result for linear stochastic equations,
which suggests the existence of non-minimal almost periodic in distribution solutions (see Theorem \ref{non-sepa} for details),
a result for nonlinear stochastic equations which weakens, in some sense, the Amerio semi-separation condition in
Theorem B (see Theorem \ref{Amerio-inherit} for details), and a result which reduces the existence of $L^2$-bounded
solutions for \eqref{SNE}, hence also for \eqref{SLE}, on the whole real line to that on the positive real line
(see Theorem \ref{L2exis} for details).

The paper is organized as follows. Section 2 is a preliminary
section in which we mainly review some fundamental properties of
almost periodic functions and introduce separation conditions for stochastic differential equations.
In Section 3, we study almost periodic solutions for linear stochastic equations under the
Favard type separation condition. In Section 4, we investigate
almost periodic solutions for nonlinear stochastic equations under the
Amerio type semi-separation condition. In Section 5, we illustrate
our results by some applications. Finally, we discuss, in Section 6, the possibility of improving
some results of Sections 3--5, i.e. we can obtain almost periodicity of solutions in
distribution sense on the path space.

Throughout the paper, we use $\R$ to denote the set of real numbers, and $\R_-=(-\infty,0]$, $\R_+=[0,+\infty)$;
we use the same symbol $|\,\cdot|$ to denote the absolute value of a number, the
Euclidian norm of a vector and the induced norm of a matrix, and the cardinality of a set; we denote by $B_r$ the
closed ball in $\R^d$ with radius $r$ centered at the origin.

\section{Preliminary}

Through this section, we assume that $(M,d)$ is a complete metric space.

\subsection{Almost periodic functions}

\begin{de}[Bohr \cite{B1}]
A continuous function $f:\R\to M$ is called {\em (Bohr) almost periodic} if for any given $\epsilon >0$,
the set
\[
T(\epsilon,f):=\{\tau\in\R:  \sup_{t\in\R} d(f(t+\tau),f(t)) <\epsilon\}
\]
is {\em relatively dense} on $\R$, i.e. there is a number $l=l(\epsilon)>0$ such that $(a, a+l)\cap T(\epsilon,f)\neq \emptyset$ for any $a\in\R$.
The set $T(\epsilon,f)$ is called the {\em set of $\epsilon$-almost periods of $f$}.
\end{de}

\begin{rem}\label{apbounded}
For given almost periodic function $f:\R\to M$, it is not hard to prove that $f$ is uniformly continuous on $\R$ and the
range $R(f)$ of $f$ is precompact, i.e. the closure of $R(f)$ is compact; see,  e.g. \cite[page 2]{LZ}.
\end{rem}

For simplicity, we follow Bochner's notation \cite{Boch62}. We denote a sequence of real numbers $\{\alpha_n\}$ by $\alpha$.
By $\alpha\subset \beta$ we mean $\alpha$ is a subsequence of $\beta$; $-\alpha$ means $\{-\alpha_n\}$; $\alpha >0$ means $\alpha_n >0$ for each $n$;
$\{\alpha+\beta\}$ means $\{\alpha_n+\beta_n\}$; $\alpha$ and $\beta$ being common subsequences of $\alpha'$ and $\beta'$ means
that $\alpha_k=\alpha'_{n(k)}$ and $\beta_k=\beta'_{n(k)}$ for the same function $n(k)$. The notation $T_\alpha f=g$ means
$g(t)=\lim_{n\to\infty}f(t+\alpha_n)$ and is written only when the limit exists; the mode of convergence will be specified
at each time when the notation is used.

The following definition of almost periodicity is due to Bochner \cite{Boch27}.

\begin{de}
A continuous function $f:\R\to M$ is called {\em (Bochner) almost periodic} if for any sequence $\alpha'$, there
exists a subsequence $\alpha\subset \alpha'$ such that $T_\alpha f$ exists uniformly on $\R$.
\end{de}

\begin{prop}[Bochner \cite{Boch27,Boch62}] \label{ap-p}
For a given continuous function $f:\R\to M$, the following statements are equivalent.
\begin{itemize}
\item[{\rm (i)}] The function $f$ is Bohr almost periodic.
\item[{\rm (ii)}] The function $f$ is Bochner almost periodic.
\item[{\rm (iii)}] For every pair of sequences $\alpha'$ and $\beta'$, there are common subsequences $\alpha\subset\alpha'$
   and $\beta\subset\beta'$ such that
\[
T_{\alpha+\beta} f = T_\alpha T_\beta f \quad \hbox{pointwise}.
\]
\item[{\rm (iv)}] For every pair of sequences $\alpha'$ and $\beta'$, there are common subsequences $\alpha\subset\alpha'$
   and $\beta\subset\beta'$ such that
\[
T_{\alpha+\beta} f = T_\alpha T_\beta f \quad \hbox{uniformly on }\R.
\]
\end{itemize}
\end{prop}

\begin{rem}
Since Bohr's almost periodicity is equivalent to Bochner's by the above result, we
will just call them almost periodicity in what follows.
\end{rem}

To study almost periodic solutions of differential equations, we need to consider uniformly almost periodic functions.

\begin{de}[Yoshizawa \cite{Y75}]\label{uap}
Let $D\subset\R^d$ be an open set. A continuous function $f:\mathbb{R}\times D\rightarrow\mathbb{R}^d$ is called {\em almost periodic in $t$
uniformly for $x\in D$} if for any $\epsilon>0$ and any compact set $S\subset D$, the set
\[
T(\epsilon,f,S):=\{\tau\in\R: \sup_{(t,x)\in\mathbb{R}\times S}|f(t+\tau,x)-f(t,x)|<\epsilon\}
\]
is relatively dense on $\R$, i.e. there is a number $l=l(\epsilon,S)>0$ such that $(a, a+l)\cap T(\epsilon,f,S)\neq \emptyset$ for any $a\in\R$.
\end{de}


Similar to almost periodic functions, we have the following result.

\begin{prop}[Yoshizawa \cite{Y75}] \label{uap-p}
Let $D\subset\R^d$ be an open set. For a given continuous function $f:\R\times D\to \R^d$, the following statements are equivalent.
\begin{itemize}
\item[{\rm (i)}] The function $f$ is almost periodic in $t$ uniformly for $x\in D$.
\item[{\rm (ii)}] For any sequence $\alpha'$, there exists a subsequence $\alpha\subset \alpha'$ such that
$T_\alpha f:=\lim_{n\to\infty} f(t+\alpha_n,x)$ exists uniformly on $\R\times S$ for any compact $S\subset D$.
\item[{\rm (iii)}] For every pair of sequences $\alpha'$ and $\beta'$, there are common subsequences $\alpha\subset\alpha'$
   and $\beta\subset\beta'$ such that for any compact $S\subset D$
\[
T_{\alpha+\beta} f = T_\alpha T_\beta f \quad \hbox{uniformly on }\R\times S.
\]
\end{itemize}
\end{prop}

For a given function $f: \mathbb{R}\times D \rightarrow\mathbb{R}^d$ almost periodic in $t$ uniformly for $x\in D$,
the {\em hull of $f$} is defined as follows:
\begin{align*}
H(f):= \{g : & \hbox{ there exists a sequence } \alpha \hbox{ such that } T_\alpha f = g \\
&\quad \hbox{ uniformly on } \mathbb{R}\times S \hbox{ for every compact set } S\subset D\}.
\end{align*}

We will need the following results in the sequel.

\begin{prop}\label{uap-p2}
Let $D\subset\R^d$ be an open set and $f:\R\times D\to \R^d$ be almost periodic in $t$ uniformly for $x\in D$.
\begin{itemize}
\item[{\rm (i)}] If a sequence $\alpha$ is such that $T_\alpha f$ exists uniformly on $\R\times S$ for any compact $S\subset D$, then
$T_\alpha f$ is almost periodic in $t$ uniformly for $x\in D$.

\item[{\rm (ii)}] Any $g\in H(f)$ is also almost periodic in $t$ uniformly for $x\in D$ and $H(g)=H(f)$.

\item[{\rm (iii)}] For any $g\in H(f)$, there exists a sequence $\alpha$ with $\alpha_n\to +\infty$ (or $\alpha_n\to -\infty$) such that
$T_{\alpha} f=g$ uniformly on $\R\times S$ for any compact $S\subset \R^d$.

\end{itemize}
\end{prop}

\begin{rem}
(i) Since we consider stochastic differential equations on $\R^d$ in this paper, i.e. $D=\R^d$ in our situation,
we will simply call a function, which is almost periodic in $t$ uniformly for $x\in \R^d$,
``uniformly almost periodic" in the sequel if there is no confusion.

(ii) Fink \cite{Fink74} and Seifert \cite{S82} introduced slightly different concepts of uniform almost periodicity;
see \cite{S82} for some discussions on their relations.
\end{rem}

\subsection{Asymptotically almost periodic functions}

\begin{de}\label{aap}
A continuous function $f:\mathbb{R_+}\to M$  is {\em asymptotically almost periodic} if there exists
an almost periodic function $p:\mathbb{R}\to M$ such that
\[
\lim_{t\rightarrow+\infty} d(f(t),p(t))=0.
\]
The function $p$ is called the {\em almost periodic part of $f$}. The asymptotically almost periodic function on $\R_-$ is defined similarly.
\end{de}

\begin{rem}
For a given asymptotically almost periodic function $f:\mathbb{R_+}\to M$, its almost periodic part is unique.
\end{rem}

\begin{prop}[Seifert \cite{S65}, Fink \cite{Fink72}]\label{aap-p}
For a given continuous function $f:\R_+ \to M$,  the following statements are equivalent.
\begin{itemize}
\item[{\rm (i)}] The function $f$ is asymptotically almost periodic.

\item[{\rm (ii)}] For any sequence $\alpha' >0$ with $\alpha'_n\rightarrow+\infty$, there exists a subsequence
$\alpha\subset\alpha'$ and a constant $d(\alpha)>0$ such that $T_{\alpha}f$ exists pointwise on $\mathbb{R}_+$ and if sequences
$\delta >0$, $\beta\subset\alpha,\gamma\subset\alpha$ are such that
\[
T_{\delta+\beta}f=h_1 \quad\hbox{and}\quad T_{\delta+\gamma}f=h_2
\]
exist pointwise on $\mathbb{R}_+$, then either $h_1\equiv h_2$ on $\R_+$ or
$\inf_{t\in\mathbb{R}_+}d(h_1(t),h_2(t))\geq 2d(\alpha)$.
\end{itemize}
\end{prop}



\subsection{Almost periodicity in distribution}

Through the paper, we assume for convenience that $(\Omega, \mathcal{F}, {P} )$ is a probability space which is rich enough
to support random variables for any given distribution on $\R^d$ or the path space $C(\R,\R^d)$, the space of $\R^d$-valued
continuous functions on $\R$. Let ${L}^{2}({P}, \mathbb R^d)$ stand for the space of
all $\mathbb R^d$-valued random variables $X$ such that
$E |X|^{2} = \int_{\Omega}|X|^{2}\rmd  P <\infty$.
For $X\in {L}^{2}({P},\mathbb R^d)$, let
$ \|X\|_{2}{:=}\left( \int_{\Omega}|X|^{2}\rmd {P} \right)^{
1/2 }$. Then ${L}^{2} ({P}, \mathbb R^d)$ is a Hilbert space
equipped with the norm $\|\cdot\|_2$. For an $\R^d$-valued stochastic process $X=\{X(t):t\in\R\}$,
if $\sup_{t\in\R}\|X(t)\|_2 <\infty$, we say $X$ is {\em $L^2$-bounded} and denote
$\|X\|_\infty := \sup_{t\in\R}\|X(t)\|_2$. Then the set of $L^2$-bounded stochastic processes is
a Banach space with the norm $\|\cdot\|_\infty$. In what follows, we also denote by $X(t)$ or $X(\cdot)$ an $\R^d$-valued stochastic process
for convenience. Let $\cal P(\R^d)$ be the space of all Borel probability measures on $\R^d$.
For a given $\R^d$-valued random variable $X$, we denote by $\cal L(X)$ the law or distribution of $X$ on
$\R^d$; for a given process $X$, by the {\em law of $X$ on $\R^d$}
we mean the $\cal P(\R^d)$-valued mapping $\mu: \R\to \cal P(\R^d), t\mapsto \cal L(X(t))$.

Next, let us introduce the concept of almost periodicity in distribution.
For the definiteness, we endow  $\cal P(\R^d)$ with the $\rho$ metric (actually other metrics are also available):
\[
\rho (\mu,\nu) :=\sup\left\{ \left| \int f \rmd \mu - \int f\rmd \nu\right|: \|f\|_{BL} \le 1
\right\}, \quad \hbox{for } \mu,\nu\in \cal P(\R^d),
\]
where $f$ are Lipschitz continuous real-valued functions on $\R^d$ with the norms
\[
\|f\|_{BL}= \|f\|_L + \|f\|_\infty, \|f\|_L=\sup_{x\neq y} \frac{|f(x)-f(y)|}{|x-y|}, \|f\|_{\infty}=\sup_{x\in\R^d}|f(x)|.
\]
A sequence $\{\mu_n\}\subset \cal P(\R^d)$ is said to {\em weakly converge to $\mu$} if $\int f \rmd\mu_n\to \int f\rmd \mu$ for
all $f\in C_b(\R^d)$, the space of all bounded continuous real-valued functions on $\R^d$.
It is well-known that  $(\cal P(\R^d),\rho)$ is a separable complete metric space
and that a sequence $\{\mu_n\}$ weakly converges to $\mu$ if and only if $\rho(\mu_n,\mu)\to 0$ as $n\to\infty$.
See \cite[Chapter 11]{Dud} for this metric $\rho$ (denoted by $\beta$ there) and its related properties.
A sequence $\{X_n\}$ of $\R^d$-valued stochastic processes is said to {\em converge in distribution} to $X$
if $\cal L(X_n(t))$ weakly converges to $\cal L (X(t))$; the mode of convergence in $t$ will be specified at each use.

\begin{de}\label{apid}
An $\mathbb R^d$-valued stochastic process $X$ is said to be {\em (asymptotically) almost periodic
in distribution} if its law on $\R^d$ is a $\cal P(\mathbb R^d)$-valued (asymptotically) almost periodic
mapping.
\end{de}

\begin{rem}
Since $(\cal P(\R^d),\rho)$ is a complete metric space, all the assertions on (asymptotic) almost periodicity
for the abstract space $(M,d)$ hold for the $\mathbb R^d$-valued stochastic processes which are (asymptotically) almost periodic
in distribution.
\end{rem}


\subsection{Stochastic differential equations and separation}

Assume that $W_1$ and $W_2$ are two independent Brownian motions on the probability space $(\Omega, \mathcal{F}, {P})$. Let
\[
W (t) =\left\{ \begin{array}{ll}
                 W_1(t), &  \hbox{ for } t\ge 0, \\
                 -W_2(-t), &  \hbox{ for } t\le 0.
               \end{array}
 \right.
\]
Then $W$ is a two-sided Brownian motion defined on the filtered probability space $(\Omega, \mathcal{F}, {P},\cal F_t )$
with $\mathcal F_t=\sigma\{W(u): u \le t\}, t\in\R$.

Consider the equation \eqref{SNE}. The triple $(X',W'), (\Omega',\cal F',P'), \{\cal F'_t:t\in\R\}$ is a {\em weak solution} of \eqref{SNE}
if $(\Omega',\cal F',P')$ is a probability space and $\{\cal F'_t:t\in\R\}$ is a filtration of sub-$\sigma$-algebras of $\cal F'$,
$W'=\{W'(t):t\in\R\}$ is an $\cal F'_t$-adapted $m$-dimensional Brownian motion and $X'=\{X'(t):t\in\R\}$ is an $\cal F'_t$-adapted $d$-dimensional process
such that
\begin{equation*}\label{defsolu}
X'(t)= X'(s)+ \int_s^t f(r,X'(r))\rmd r
+\int_s^t g(r,X'(r))\rmd W'(r)
\end{equation*}
for all $t\ge s$ and each $s\in\mathbb{R}$ almost surely. The weak solution $(X',W'), (\Omega',\cal F',P'), \{\cal F'_t:t\in\R\}$ is a {\em strong solution}
if for given $t_0\in\R$, there exists a measurable function $h$ such that $X'(\cdot)=h(X'(t_0),W'(\cdot))$ on $\R$ almost surely.
For strong/weak solutions on the positive real line, see \cite{IW} or \cite{KS} for details; when the coefficients of \eqref{SNE}
are globally Lipschitz and of linear growth, see Remark \ref{comm-key} for some properties of strong/weak solutions of \eqref{SNE}
on $\R$.

For the Cauchy problem of \eqref{SNE} on the positive real line, it is well-known that the
pathwise uniqueness
implies uniqueness in the sense of probability law on the path space which we simply call ``weak uniqueness",
see e.g. \cite[\S IV.1]{IW}; in the meantime, we note that the weak uniqueness implies the uniqueness of law on $\R^d$.


Consider \eqref{SNE}. To emphasize explicitly the coefficients of \eqref{SNE}, we also call it {\em equation $(f,g)$}. For given $r>0$, we
introduce the following notations:
\begin{align*}
& \cal B_r:=\{X\in L^2(P,\R^d): \|X\|_2 \le r\}, \quad \cal D_r:=\{ \mu\in\cal P(\R^d): \int_{\R^d} |x|^2 \rmd \mu(x) \le r^2 \}, \\
&\cal B_r^{\eqref{SNE}}=\cal B_r^{(f,g)} :=\{X(\cdot): (X,W) \hbox{ weakly solves equation } (f,g)  \hbox{ on } \R \\
&\qquad\qquad\qquad\qquad \hbox{ on some filtered probability space for some } W \hbox{ and } \|X\|_\infty\le r \},\\
&\cal D_r^{\eqref{SNE}}=\cal D_r^{(f,g)} :=\{\mu: \mu(\cdot) =\cal L (X(\cdot)) \hbox{ for some } X \in \cal B_r^{(f,g)} \}.
\end{align*}

\begin{de}
If $\cal B_r^{\eqref{SNE}}$ is non-empty for some $r>0$, then $\lambda:=\inf_{X\in \cal B_r^{\eqref{SNE}}} \|X\|_{\infty}$ is called
the {\em minimal value of \eqref{SNE}}; if $X_0\in \cal B_r^{\eqref{SNE}}$ and $\|X_0\|_\infty =\lambda$, then $X_0$ is a
{\em minimal (weak) solution of \eqref{SNE}}.
\end{de}

\begin{de}
(i) Assume that the coefficients $A,f,B_i,g_i$ of \eqref{SLE} are almost periodic. If there is a sequence $\alpha$ such that $T_\alpha A=\tilde A$,
$T_\alpha f=\tilde f$, $T_\alpha B_i=\tilde B_i$ and $T_\alpha g_i=\tilde g_i$ uniformly on $\R$ for $i=1,\ldots,m$, then
the equations
\begin{align*}
\rmd X = (\tilde A(t)X+\tilde f(t)) \rmd t + \sum_{i=1}^m(\tilde B_i(t)X +\tilde g_i(t)) \rmd W_i,
\end{align*}
and
\begin{align*}
\rmd X = \tilde A(t)X \rmd t + \sum_{i=1}^m \tilde B_i(t)X  \rmd W_i
\end{align*}
are called  {\em hull equation of \eqref{SLE}} and {\em homogeneous hull equation of \eqref{SLE}}, respectively.
\\
(ii) Assume that $f,g$ in \eqref{SNE} are uniformly almost periodic.
The equation $(\tilde f,\tilde g)$ is called a {\em hull equation of \eqref{SNE}}, denoted by $(\tilde f,\tilde g)\in H(f,g)$,
if there exists a sequence $\alpha$ such that
$T_\alpha f=\tilde f$ and $T_\alpha g=\tilde g$, also denoted by $T_\alpha(f,g)=(\tilde f,\tilde g)$, uniformly on
$\R\times S$ for any compact subset $S\subset\R^d$.
\end{de}

\begin{de}\label{Fav-SLE}
We say that the {\em Favard (separation) condition  holds for \eqref{SLE}} if for any homogeneous hull equation corresponding to \eqref{SLE}
\begin{equation}\label{SLEh}
\rmd X = \tilde A(t)X \rmd t + \sum_{i=1}^m \tilde B_i(t)X\rmd W_i,
\end{equation}
every nontrivial $L^2$-bounded weak solution $X$ of \eqref{SLEh} on $\R$ satisfies $\inf_{t\in\R} \|X(t)\|_2 >0$.
\end{de}

\begin{rem}\label{Fav-u}
Note that if any nontrivial deterministic solution $x(t)$ of the equation $\dot x =A(t)x$ satisfies $\inf_{t\in\R} |x(t)|>0$, then
we have $\inf_{t\in\R}\|X(t)\|_2 >0$ for any nontrivial $L^2$-bounded stochastic process $X$ which satisfies the same equation.
The converse is obviously true. Therefore, the
Favard separation condition in Definition \ref{Fav-SLE} is a natural generalized version of the usual one mentioned in the
Introduction.
\end{rem}

\begin{de}
Assume that $f,g$ in \eqref{SNE} are uniformly almost periodic.
We say that the {\em Amerio positive (resp. negative) semi-separation condition holds for \eqref{SNE}} in $\cal D_r$
if any hull equation $(\tilde f, \tilde g)$ of \eqref{SNE}
only admits positive (resp. negative) semi-separated in distribution solutions in $\cal B_r$; that is, for any $\mu \in \cal D_r^{(\tilde f, \tilde g)}$,
there is a constant $d(\mu)>0$, called {\em separation constant}, such that $\inf_{t\ge0}\rho(\mu(t),\nu(t))\ge d(\mu)$ (resp. $\inf_{t\le 0}\rho(\mu(t),\nu(t))\ge d(\mu)$)
for any other $\nu\in \cal D_r^{(\tilde f, \tilde g)}$.
\end{de}

\begin{de}
A property $P$ is called {\em negative semi-separating in $\cal D_r^{\eqref{SNE}}$}
if for any distinct $\mu,\nu\in\cal D_r^{\eqref{SNE}}$ which satisfy $P$, there exists a constant $d^{\mu,\nu}>0$
such that $\inf_{t\in\R_-}\rho(\mu(t),\nu(t))\ge d^{\mu,\nu}$.
\end{de}

\begin{de}
Assume that $f,g$ in \eqref{SNE} are uniformly almost periodic.
A property $P$ is {\em inherited in distribution in $\cal D_r$} if $\mu \in\cal D_r^{(f,g)}$ has
property $P$ with respect to the elements of $\cal D_r^{(f,g)}$, $(\tilde f,\tilde g)\in H(f,g)$ with
$T_{\alpha}(f,g)=(\tilde f, \tilde g)$ and $T_{\alpha}\mu=\nu$ uniformly on compact intervals for some sequence $\alpha$,
then $\nu$ also has property $P$ with respect to the elements of $\cal D_r^{(\tilde f,\tilde g)}$.
\end{de}

\section{Favard separation for linear stochastic equations}
\setcounter{equation}{0}

The following result, which simply says that limits of solutions are solutions of the limit equation
in distribution sense, is a key ``lemma" for what follows and interesting on its own rights,
so we state it as a theorem.

\begin{thm}\label{key}
Consider the following family of It\^o stochastic equations on $\R^d$
\[
\rmd X = f_n(t,X)\rmd t + g_n(t,X)\rmd W,\quad n=1,2,\cdots, 
\]
where $f_n$ are $\R^d$-valued, $g_n$ are $(d\times m)$-matrix-valued, and
$W$ is a standard $m$-dimensional Brownian motion.
Assume that $f_n, g_n$ satisfy the conditions of global Lipschitz and linear growth with common Lipschitz and linear growth constants; that is,
there are constants $L$ and $K$, independent of $t\in\mathbb R$ and $n\in\N$, such that for all $x,y\in\mathbb R^d$
\begin{align*}
&|f_n(t,x)-f_n(t,y)|\vee |g_n(t,x)-g_n(t,y)|\le L|x-y|,\\
& |f_n(t,x)|\vee |g_n(t,x)|\le K (1+|x|),
\end{align*}
where $a\vee b=\max\{a,b\}$. Assume further that $f_n\to f$, $g_n\to g$ pointwise on $\mathbb R\times \mathbb R^d$ as $n\to\infty$ and that
$X_n\in \cal B_{r_0}^{(f_n,g_n)}$ for some constant $r_0$, independent of $n$. Then there is a subsequence of $\{X_n\}$ which
converges in distribution, uniformly on compact intervals, to some $X\in \cal B_{r_0}^{(f,g)}$.
\end{thm}

\begin{proof}
Fix $n\in\N$. For given bounded interval $[a,b]\subset \R$ and $X_n\in \cal B_{r_0}^{(f_n,g_n)}$ with $(X_n,W_n)$ weakly
solving equation $(f_n, g_n)$ on some filtered probability space, by Cauchy-Schwarz inequality and It\^o's isometry we have
for any $a \le s \le t \le b$
\begin{align}\label{l2c}
E |X_n(t)-X_n(s)|^2  & = E \left| \int_s^t f_n(r,X_n(r))\rmd r + \int_s^t g_n(r,X_n(r)) \rmd W_n(r) \right|^2 \nonumber\\
& \le 2 (t-s) \int_s^t E |f_n(r,X_n(r))|^2 \rmd r + 2 \int_s^t E |g_n(r,X_n(r))|^2 \rmd r \nonumber\\
& \le 2 (t-s) K^2\int_s^t  E(1+|X_n(r)|)^2 \rmd r + 2 K^2\int_s^t  E(1+|X_n(r)|)^2 \rmd r \nonumber\\
& \le 4 (t-s) K^2 \int_s^t  E(1+|X_n(r)|^2) \rmd r + 4 K^2\int_s^t  E(1+|X_n(r)|^2) \rmd r \nonumber\\
& \le 4K^2(\|X_n\|_\infty^2 +1)(b-a+1)(t-s)\nonumber\\
& \le 4K^2 (r_0^2 +1)(b-a+1)(t-s).
\end{align}
Note that the estimate \eqref{l2c} is uniform in $X_n,n\in \mathbb N$.

It follows from Chebyshev's inequality that, for any $X\in \cal B_{r_0}$ and $C\in\mathbb{R}_+$, we have
\[
{P}\{|X|> C\}\le \frac{{E}|X|^2}{C^2}\leq\frac{r_0^2}{C^2}.
\]
So for given $\epsilon>0$, there is a compact set $K_{\epsilon}:=B_C$, a closed ball in $\R^d$, such that
$${P}\{X\in\mathbb{R}^d\backslash K_{\epsilon}\}\leq\epsilon$$
by choosing $C >0$ large enough.
By the Prohorov's theorem \cite{Pro}, $\cal D_{r_0}$ is contained in a compact set in $\mathcal{P}(\mathbb{R}^d)$; actually the Fatou's lemma
yields that $\cal D_{r_0}$ is closed and hence compact in $\mathcal{P}(\mathbb{R}^d)$.
It follows from \eqref{l2c} that the sequence $\{X_n\}$, regarded as continuous mappings from $[a,b]$ to $L^2(P,\R^d)$,
is equi-continuous. Denote $\mu_n(\cdot)=\cal L(X_n(\cdot)): [a,b]\to \cal P(\R^d)$, the law of $X_n(\cdot)$ on $\R^d$.
Then since $L^2$-continuity implies
continuity in distribution,  the sequence $\{\mu_n\}$ is equi-continuous.
Applying a general version of Arzela-Ascoli Theorem (see, e.g. \cite[Theorem 7.17]{Kel}), we obtain a subsequence of $\{\mu_n\}$,
still denote by $\{\mu_n\}$, which converges uniformly on $[a,b]$. Since the interval $[a,b]$ is arbitrary, by the diagonal
method there is a further subsequence, still denote by $\{\mu_n\}$, such that $\mu_n:\R\to\cal P(\R^d)$
converges to a function $\mu:\R\to\cal P(\R^d)$, uniformly on any compact interval.

In the remaining part of the proof, we prove that the limit $\mu$ is the law of some $L^2$-bounded solution $X$ of
the equation $(f,g)$ with $\|X\|_\infty\le r_0$,
so the theorem is proved.

For the given bounded interval $[a,b]$, since
$\mu_n(a)\to\mu(a)$ as $n\to\infty$, by the Skorohod representation theorem there is a probability space
$(\tilde\Omega,\tilde{\cal F},\tilde P)$ and random variables $\{\tilde X_n(a) \}_{n=1}^\infty$, $\tilde X(a)$ defined on it so that
$\cal L(\tilde X_n(a)) =\cal L(X_n(a))$, $\cal L(\tilde X(a)) = \mu(a)$ and $\tilde X_n(a)\to \tilde X(a)$ almost surely as $n\to\infty$.
We consider the equation $(f_n,g_n)$ with a
common Brownian motion $W$ on the probability space $(\tilde\Omega,\tilde{\cal F},\tilde P)$. Then since the coefficients $f_n, g_n$
satisfy the global Lipschitz and linear growth conditions, by the classical
approximation theorem (see, e.g. \cite[p54, Theorem 3]{GS}), we have
\begin{equation}\label{appr}
\sup_{t\in [a,b]} |\tilde X_n(t)-\tilde X(t)| \to 0 \quad \hbox{in probability as } n\to \infty,
\end{equation}
where $\tilde X_n(\cdot)$ and $\tilde X(\cdot)$ are strong solutions on $[a,b]$ of equations $(f_n,g_n)$ and $(f,g)$ with the common
Brownian motion $W$ and initial values $\tilde X_n(a)$ and $\tilde X(a)$, respectively. This implies that
\[
\mu_n(t)=\cal L(X_n(t)) = \cal L(\tilde X_n(t)) \to \cal L (\tilde X(t))
\]
uniformly on $[a,b]$, where $\cal L(X_n(t)) = \cal L(\tilde X_n(t))$ holds since the weak uniqueness for equation $(f_n,g_n)$ on $[a,b]$ holds and
weak uniqueness implies uniqueness of laws on $\R^d$.
On the other hand, $\mu_n(\cdot)\to \mu(\cdot)$ on $[a,b]$.
So this enforces that $\mu(t)=\cal L(\tilde X(t))$, $t\in [a,b]$.
We may restart from $b$ and repeat the above procedure. In this way, we have proved that $\mu$ is the law of some
solution of the equation $(f,g)$ on the half line $[a, \infty)$.

Next we will construct, by Kunita's stochastic flow method \cite{Kun}, a strong solution on $(-\infty,a]$ of the equation $(f,g)$ with the above common $W$
so that its law on $\R^d$ coincides with $\mu$ on $(-\infty,a]$.  Since $f$ and $g$ satisfy the global Lipschitz condition, by \cite[Theorem 2.4.3]{Kun}
we know that the solution mapping $\Phi_{s,t}(\cdot,\omega):\R^d\to\R^d$ of the equation $(f,g)$ is a homeomorphism of $\R^d$ for any $s<t$ and almost all $\omega$.
For given $c<a$, we take $\tilde X(c,\omega)= \Phi_{c,a}^{-1}(\tilde X(a,\omega),\omega)$ for each $\omega$, i.e. the inverse
image of $\tilde X(a)$ at ``time" $c$. So if we consider the equation $(f,g)$ on $[c,a]$ with initial value $\tilde X(c)$,
then the value of the solution at ``time" $a$ is exactly $\tilde X(a)$. In the same way, we take $\tilde X_n(c,\omega)= (\Phi_{c,a}^{n}(\tilde X_n(a,\omega),\omega))^{-1}$,
with $\Phi^{n}$ being the solution mapping of the equation $(f_n, g_n)$. Then the convergence of $\tilde X_n(a)$ to $\tilde X(a)$ implies that
of $\tilde X_n(c)$ to $\tilde X(c)$ since $\Phi^{n}$ is a homeomorphism and $\Phi^n_{c,a}\to\Phi_{c,a}$ as $n\to\infty$
by the above mentioned approximation theorem.
The same argument as that on $[a,b]$ shows that $\mu(\cdot)$ is the law of $\tilde X(\cdot)$ on the interval $[c,a]$. By repeating
the procedure, it follows that $\mu(\cdot)$ is the law of $\tilde X(\cdot)$ on $(-\infty,a]$ and hence on $\R$.

Since $\tilde X_n(t)$ converges in probability to $\tilde X(t)$ for each $t\in\R$, the Fatou's lemma and the fact $\|\tilde X_n\|_\infty \le r_0$ imply that
\[
E |\tilde X(t)|^2 \le \liminf_{n} E |\tilde X_n(t)|^2 \le r_0^2.
\]
That is, $\|\tilde X\|_\infty \le r_0$.

Finally, we replace $\tilde X(a)$ on $(\tilde\Omega,\tilde{\cal F},\tilde P)$ by a random variable
$X(a)$ on $(\Omega,{\cal F}, P)$ with the same law on $\R^d$, and denote the corresponding solution (the existence and
uniqueness is guaranteed by the Lipschitz and linear growth conditions) of equation $(f,g)$ by $X(t)$. Then this solution $X(t)$
admits the same distribution on $\R^d$ as that of $\tilde X(t)$ by weak uniqueness for the equation $(f,g)$,
and we also have $\|X\|_\infty \le r_0$. This $X(t)$ is what we look for. The proof is complete.
\end{proof}

\begin{rem}\label{comm-key}
(i) The above theorem is nontrivial since we consider solutions on the whole real line instead of the usual case where we
consider solutions of the Cauchy problem on a positive finite interval or the positive real line.

(ii) For given Brownian motion on some probability space, it follows from the proof of Theorem \ref{key} that
each solution of \eqref{SNE} on $\R$ is determined by the
``initial value" at time $0$ or at any given ``time" $a\in\R$, under the global Lipschitz and linear growth conditions.
That is, a weak solution on $\R$ is actually a strong solution on $\R$ in this case; this is similar to the usual case of solutions
on the positive real line. Therefore, on some occasions we will not distinguish weak or strong solutions (just call them solutions)
in what follows since the equations we consider in this paper satisfy the conditions of global Lipschitz and linear growth.
In this case, for convenience we assume that the probability space $(\Omega,\cal F,P)$, the Brownian motion $W$ and the filter $\cal F_t$
are fixed, as pointed out in the Introduction.


(iii) We can also observe from the proof of Theorem \ref{key} that if we only consider the law on $\R^d$ of solutions
of \eqref{SNE} on $\R$, then the law on $\R^d$ is determined
by the ``initial law" at time $0$ or at any given ``time" $a\in\R$ by the weak uniqueness on the positive real line and
the fact that the Kunita's stochastic flow theorem holds under the global Lipschitz and linear growth conditions.
That is, for any given two random variables at $0$ or any $a\in\R$
with the same law, the solutions on $\R$ they determine share the same law on $\R^d$. Actually the stronger result holds:
they share the same law on the path space.


\end{rem}

\begin{lem}\label{exis}
Assume that the coefficients $f,g$ of the equation \eqref{SNE} satisfy the conditions of global Lipschitz and linear growth,
and that there is an $L^2$-bounded solution for \eqref{SNE}. Then \eqref{SNE} admits a minimal
solution.
\end{lem}

\begin{proof}
Denote $\lambda$ as the minimal value of \eqref{SNE}.
Take a sequence $\{X_n\}$ of $L^2$-bounded solutions of \eqref{SNE} such that
\begin{equation*}
\|X_n\|_\infty \le \lambda + \frac1n.
\end{equation*}
Then it follows from Theorem \ref{key} that there is a subsequence of $\{X_n\}$ which converges in distribution to some solution $X$ of \eqref{SNE},
with $\|X\|_\infty \le \lambda$.
This limit solution is a minimal solution of \eqref{SNE}.
\end{proof}

\begin{lem}\label{homo}
Consider the homogeneous linear equation corresponding to \eqref{SLE} on $\R^d$
\begin{equation*}
\rmd X = A(t)X \rmd t + \sum_{i=1}^m B_i(t)X \rmd W_i.
\end{equation*}
Assume that $A,B_i$ are almost periodic and that $Y(t)$ is an $L^2$-bounded solution of the above equation on $\R$ which
is almost periodic in distribution. Then we have the following alternative:
\[
\inf_{t\in\R} \|Y(t)\|_2 >0 \quad \hbox{or} \quad Y(t)=0 \hbox{ for all } t\in\R \hbox{ a.s.}
\]
\end{lem}

\begin{proof}
We only need to show that $\inf_{t\in\R} \|Y(t)\|_2 =0$ implies $Y(t)= 0$ almost surely for all $t\in\R$, which implies $Y(t)= 0$
for all $t\in\R$ almost surely since $Y(t)$ is a continuous process. So let us assume
$\inf_{t\in\R} \|Y(t)\|_2 =0$, then there exists a sequence $\alpha'=\{\alpha'_n\}$ such that $Y(\alpha'_n)\to 0$ in
$L^2(P,\R^d)$. It follows from Proposition \ref{ap-p} (iv) (by choosing $\beta'=-\alpha'$ there)
that there exits a subsequence $\alpha\subset\alpha'$ so that
\begin{align*}
 T_\alpha A(t)= \tilde A(t), \quad T_\alpha B_i(t) =\tilde B_i(t),
 \quad T_{-\alpha}\tilde A(t)=  A(t), \quad T_{-\alpha}\tilde B_i(t) = B_i(t)
\end{align*}
and
\begin{align*}
 T_\alpha \mu(t) = \tilde \mu(t), \quad T_{-\alpha}\tilde \mu(t) = \mu(t)
\end{align*}
uniformly on $\R$, where $\mu(\cdot)$ is the law on $\R^d$ of the solution $Y(\cdot)$. By the proof of Theorem \ref{key},
the limit $\tilde \mu(\cdot)$ is the law of some solution $\tilde Y(\cdot)$
of the limit equation
\[
\rmd X = \tilde A(t)X \rmd t + \sum_{i=1}^m \tilde B_i(t)X \rmd W_i.
\]

Note that $\tilde\mu(0)=T_\alpha\mu(0)=\lim_{n\to\infty}\mu(\alpha_n)=\delta_0$ weakly, with $\delta_0$ being the Dirac measure at $0$.
So we have $\tilde Y(0) = 0$ almost surely, and hence $\tilde Y(t) =0 $ almost surely for $t\ge 0$
by the uniqueness of solutions; then by the
Kunita's stochastic flow method, as in the proof of Theorem \ref{key} again, we have $\tilde Y(t)= 0$ almost surely for $t\in\R_-$
and hence $\tilde \mu(t)\equiv \delta_0$ for $t\in\R$.
So we have $\mu(t)=T_{-\alpha}\tilde\mu(t)=\delta_0$ for each $t\in\R$. Therefore, $Y(t)= 0$ almost surely on $\R$. The proof is complete.
\end{proof}

\begin{rem}\label{refa}
Consider the linear equation of the form on $\R^d$
\begin{equation}\label{sim}
\rmd X = (AX +f(t) )\rmd t + g(t) \rmd W,
\end{equation}
where $A$ is a constant matrix, and $f,g$ are almost periodic.
Since any non-trivial deterministic bounded solution of $\dot x=Ax$ is almost periodic by \cite[Theorem 5.3]{Fink74},
it follows from Lemma \ref{homo} that these solutions are separated from $0$. So
the Favard condition holds for \eqref{sim} by Remark \ref{Fav-u}.
\end{rem}

\begin{lem}\label{Fav}
Assume that the Favard condition holds for the linear equation \eqref {SLE},
then, for given Brownian motion $W$ on the filtered probability space $(\Omega,\cal F,P,\cal F_t)$, there is
at most one strong minimal solution for any hull equation of \eqref {SLE}.
\end{lem}

\begin{proof}
If the assertion is not true, then there exists a hull equation
\begin{equation}\label{h}
\rmd X= (\tilde A(t)X + \tilde f(t)) \rmd t + \sum_{i=1}^m(\tilde B_i(t)X + \tilde g_i(t))\rmd W_i
\end{equation}
so that $X_1$ and $X_2$ are both minimal solutions of the above equation with the common minimal value $\lambda$.
Then $(X_1-X_2)/2$ is a nontrivial $L^2$-bounded solution of the corresponding homogeneous hull equation
\[
\rmd X= \tilde A(t)X  \rmd t + \sum_{i=1}^m\tilde B_i(t)X \rmd W_i.
\]
But the Favard condition yields that there is a constant $\eta>0$ so that
\[
\inf_{t\in\R} \frac12 \|X_1(t)-X_2(t)\|_2 \ge \eta.
\]
It follows from the parallelogram formula that for any $t\in\R$
\[
\|\frac12 (X_1(t) + X_2(t))\|^2_2 + \|\frac12 (X_1(t) - X_2(t))\|^2_2 = \frac12 (\|X_1(t)\|^2_2 + \|X_2(t)\|^2_2)
\le \lambda^2.
\]
So $\|(X_1 + X_2)/2\|_\infty < \lambda$, the minimal value. This is a contraction since $(X_1 + X_2)/2$ is
an $L^2$-bounded solution of \eqref{h}.
\end{proof}

\begin{rem}\label{rem-Fav}
By Lemma \ref{Fav} and Remark \ref{comm-key} (ii)-(iii), it follows that, for any given hull equation of \eqref{SLE}, all the weak minimal
solutions (if they exist) of it share the same law on the path space and hence on $\R^d$ if the Favard condition holds for \eqref{SLE}.
\end{rem}

\begin{lem}\label{limit}
Assume that $f,g$ in \eqref{SNE} are uniformly almost periodic and satisfy global Lipschitz condition with Lipschitz constants
independent of $t$. Then any hull equation of \eqref{SNE}
admits the same minimal value as that of \eqref{SNE}.
\end{lem}

\begin{proof}
Firstly note that the coefficients $f,g$ of \eqref{SNE} satisfy the linear growth condition since $f(\cdot,0),g(\cdot,0)$ are bounded
on $\R$ by Remark \ref{apbounded}. Assume that $\varphi$ is a minimal solution of \eqref{SNE},
i.e. $\|\varphi\|_\infty = \lambda$, the minimal value of \eqref{SNE}.
Then we have for any $s<t$
\[
\varphi(t) = \varphi(s) + \int_s^t f(r,\varphi(r))\rmd r + \int_s^t g(r,\varphi(r)) \rmd W(r)
\]
for some Brownian motion $W$. Consider the hull equation $(\tilde f,\tilde g)$ with $T_\alpha(f,g)=(\tilde f,\tilde g)$.
Denote $\varphi_n(\cdot):= \varphi(\cdot+\alpha_n)$, $f_n(\cdot,\cdot):= f(\cdot+\alpha_n,\cdot)$, $g_n(\cdot,\cdot):= g(\cdot+\alpha_n,\cdot)$,
and $W_n(\cdot):= W(\cdot+\alpha_n)-W(\alpha_n)$. Note that $f_n,g_n$ are uniformly almost periodic and globally Lipschitz
with the same Lipschitz constants as that of $f,g$, $W_n$ are standard
Brownian motions, and that $f_n\to \tilde f$, $g_n\to \tilde g$ uniformly on $\R\times S$ as $n\to\infty$ for any compact subset $S\subset\R^d$.
It is clear that $\varphi_n$ satisfies the following equation for any $s<t$
\[
\varphi_n(t) = \varphi_n(s) + \int_s^t f_n(r,\varphi_n(r))\rmd r + \int_s^t g_n(r,\varphi_n(r)) \rmd W_n(r).
\]

By Theorem \ref{key},
there is a subsequence of $\{\varphi_n\}$ which we still denote by the
sequence itself so that $\varphi_n$ converges in distribution, uniformly on compact intervals,
to some $\tilde\varphi$ as $n\to\infty$ which satisfies
the hull equation on $\R$, i.e. for any $s<t$
\[
\tilde\varphi(t) = \tilde\varphi(s) + \int_s^t \tilde f(r,\tilde\varphi(r))\rmd r + \int_s^t \tilde g(r,\tilde\varphi(r)) \rmd \tilde W(r)
\]
for some Brownian motion $\tilde W$.
The Fatou's lemma implies that $\|\tilde \varphi\|_\infty \le \|\varphi\|_\infty = \lambda$. Hence $\tilde\lambda \le \lambda$, where $\tilde\lambda$
denotes the minimal value of the hull equation $(\tilde f,\tilde g)$.

Conversely, by the property of uniform almost periodic functions, we have $T_{-\alpha} \tilde f = f$ and $T_{-\alpha} \tilde g = g$.
By the symmetry, we have $\lambda\le \tilde\lambda$. Therefore, $\lambda=\tilde\lambda$. The proof is complete.
\end{proof}

\begin{cor}\label{mini}
Consider \eqref{SNE} and assume that the assumptions of Lemma \ref{limit} hold. Assume further that $\varphi$ is a
minimal solution of \eqref{SNE}, and that the sequence $\alpha$
satisfies $T_\alpha(f,g)=(\tilde f,\tilde g)$ and $T_\alpha \varphi$ converges in distribution, uniformly on compact intervals,
to some solution $\tilde \varphi$ of
equation $(\tilde f,\tilde g)$. Then $\tilde\varphi$ is a minimal solution of equation $(\tilde f,\tilde g)$.
\end{cor}

\begin{proof}
It is immediate from the proof of Lemma \ref{limit}.
\end{proof}

\begin{lem}\label{uniap}
Assume that each hull equation of \eqref{SNE} admits a unique minimal solution in distribution sense, i.e. all the minimal
solutions of the given hull equation possess the same law on $\R^d$. Then these minimal solutions are almost periodic in distribution.
\end{lem}

\begin{proof}
For given $(\tilde f,\tilde g)\in H(f,g)$ and arbitrary sequences $\alpha'$ and $\beta'$, by the property of uniform almost periodic
functions there exist common subsequences $\alpha,\beta$ of $\alpha',\beta'$ so that
\[
T_{\alpha+\beta} \tilde f = T_\alpha T_\beta \tilde f,\quad  T_{\alpha+\beta} \tilde g = T_\alpha T_\beta \tilde g
\]
uniformly on $\R\times S$ for any compact subset $S$ of $\R^d$. Assume that $\tilde\varphi$ is a minimal solution of the equation
$(\tilde f,\tilde g)$ whose law on $\R^d$ is $\tilde\mu$.

By the proof of Theorem \ref{key}, there exist common subsequences of $\alpha$ and $\beta$, which we still denote by $\alpha$ and $\beta$,
so that $T_{\alpha+\beta}\tilde \mu$ and $T_\alpha T_\beta \tilde\mu$ exist, uniformly on compact intervals, and they are laws of
processes $\varphi_1$ and $\varphi_2$, which are solutions of the equations $(T_{\alpha+\beta} \tilde f, T_{\alpha+\beta} \tilde g)$ and
$(T_\alpha T_\beta \tilde f, T_\alpha T_\beta \tilde g)$, respectively.
That is, $\varphi_1$ and $\varphi_2$ satisfy the same equation. By Corollary \ref{mini},
both $\varphi_1$ and $\varphi_2$ are minimal solutions of the equation
$(T_{\alpha+\beta} \tilde f, T_{\alpha+\beta} \tilde g)$. But each hull equation admits a unique minimal solution
in distribution sense, which enforces
$\cal L(\varphi_1) = \cal L(\varphi_2)$ and hence $T_{\alpha+\beta}\tilde \mu= T_\alpha T_\beta \tilde\mu$ .
That is, $\tilde \varphi$ is almost periodic in distribution by Proposition \ref{ap-p}. The proof is complete.
\end{proof}


The following result is Theorem A in the Introduction.

\begin{thm}\label{Favard}
Consider \eqref{SLE} with the coefficients $A, B_1,\ldots, B_m$, $f,g_1,\ldots,g_m$ being almost periodic.
Assume further that \eqref{SLE} admits an $L^2$-bounded solution and that the Favard condition holds
for \eqref{SLE}. Then \eqref{SLE} admits an almost periodic in distribution solution.
\end{thm}

\begin{proof}
It follows from Lemma \ref{exis} and Corollary \ref{mini} that each hull equation of \eqref{SLE}
admits minimal solutions, and by Lemma \ref{Fav} and Remark \ref{rem-Fav} each hull equation admits a unique
minimal solution in distribution sense. So the theorem follows from Lemma \ref{uniap}.
\end{proof}

\begin{rem}\label{BN}
It follows from Remark \ref{refa} and Theorem \ref{Favard} that the existence of $L^2$-bounded solutions of
\eqref{sim} implies that it admits an almost periodic in distribution solution. This can be regarded as a stochastic version
of Bohr-Neugebauer type result, mentioned in the Introduction.
\end{rem}

\begin{cor}\label{thAco}
Consider the equation of the form on $\R^d$
\begin{equation}\label{exa}
\rmd X= [A(t)X + f(t)] \rmd t + g(t) \rmd W,
\end{equation}
where $A,f,g$ are almost periodic functions. If the corresponding deterministic equation
\begin{equation}\label{here0}
\dot x= A(t)x + f(t)
\end{equation}
satisfies the Favard condition in usual sense (mentioned in the Introduction) and \eqref{exa} admits an $L^2$-bounded solution, then
\eqref{exa} admits an almost periodic in distribution solution.
\end{cor}

\begin{proof}
Note that
the homogeneous equation corresponding to \eqref{exa} is the same
as that of the deterministic equation \eqref{here0},
so the Favard condition holds for \eqref{exa} by Remark \ref{Fav-u}.
The result now follows from Theorem \ref{Favard}.
\end{proof}

Finally we give a result, which confirms that there may be other almost periodic in distribution solutions
besides minimal ones.

\begin{thm}\label{non-sepa}
Consider the linear equation on $\R^d$
\begin{equation}\label{lincon}
\rmd X = (AX + f(t) )\rmd t + g(t) \rmd W,
\end{equation}
where $A$ is a constant matrix, $f,g$ are almost periodic, and $W$ is a given $m$-dimensional Brownian motion. If $X$ is a strong
$L^2$-bounded solution of \eqref{lincon} on $\R$
so that $X(\tau)-X_0(\tau)$ is independent of $X_0(\tau)$ and $W$
for some $\tau\in\R$, where $X_0$ is the strong minimal solution of \eqref{lincon}. Then $X$ is almost periodic in distribution.
\end{thm}

\begin{proof}
By Remark \ref{refa} we know that the Favard condition holds for \eqref{lincon},
so it follows from Lemmas \ref{exis} and \ref{Fav} that there is a unique strong minimal solution $X_0$ for equation
\eqref{lincon}, which is almost periodic in distribution by Lemma \ref{uniap}.  Let $Y(\cdot)= X(\cdot)-X_0(\cdot)$. Then $Y$ is
an $L^2$-bounded solution of the equation $\dot x = Ax$ on $\R$, which is almost periodic in $L^2$-sense  by \cite[Theorem 5.3]{Fink74}\footnote{
Note that the argument there still applies when the initial value is replaced by an $L^2$-random variable.},
i.e. the mapping $t\mapsto Y(t)$ is an almost periodic $L^2(P,\R^d)$-valued mapping. By the Bochner's definition for almost periodicity, it follows that
$Y(\cdot)$ is almost periodic in distribution since $L^2$-convergence implies convergence in distribution.

Since $X_0$ is the solution of \eqref{lincon} with ``initial value" $X_0(\tau)$, by the strong solution theorem of Yamada-Watanabe
(see \cite{YW} or \cite[Theorem IV.1.1]{IW}),
\begin{equation}\label{strong}
X_0(\cdot)=F(X_0(\tau),W) \hbox{ for some measurable function } F
\end{equation}
and $t\ge \tau$. Since the coefficients of \eqref{lincon} are globally Lipschitz in $x$,
Kunita's stochastic flow theorem implies that
we may regard that \eqref{strong} holds
for all $t\in\R$. Similarly, for the equation $\dot x=Ax$, we have $Y(\cdot) = G(Y(\tau))$ for some measurable function $G$.
Since the random variable $Y(\tau)$ is independent of $X_0(\tau)$ and $W$, $Y(\cdot)$ is independent of $X_0(\cdot)$.
In particular, $X_0(t)$ is independent of $Y(t)$ for each $t\in\R$. So we have
\[
\cal L(X(t)) = \cal L(X_0(t)) * \cal L(Y(t)) \qquad \hbox{for each } t\in\R.
\]

Denote $\mu_1(t)= \cal L(X_0(t))$ and $\mu_2(t)=\cal L(Y(t))$ for each $t\in\R$. For arbitrary sequences $\alpha'$ and $\beta'$,
it follows from the almost periodicity of $\mu_1$ and $\mu_2$ that there are common subsequences $\alpha\subset\alpha'$
and $\beta\subset\beta'$ such that
\[
T_{\alpha+\beta}\mu_i(t) = T_\alpha T_\beta \mu_i(t)  \qquad \hbox{for each } t\in\R, i=1,2.
\]
Since the convolution of probability measures is continuous (see, e.g. \cite[Theorem 9.5.9]{Dud}), it follows that
\[
T_{\alpha+\beta} [\mu_1(t) * \mu_2(t)] = T_\alpha T_\beta [\mu_1(t) * \mu_2(t)]  \qquad \hbox{for each } t\in\R.
\]
That is, $\cal L(X(\cdot))$ is almost periodic by Proposition \ref{ap-p}. The proof is complete.
\end{proof}


\section{Amerio separation for nonlinear stochastic equations}

In this section, we consider the nonlinear equation \eqref{SNE}. Firstly, let us state the
following standing hypothesis which is used frequently in the sequel:

{\bf(H)} Assume that $f(t,x)$ is an $\R^d$-valued uniformly almost periodic function,
$g(t,x)$ is a $(d\times m)$-matrix-valued uniformly almost periodic function,
and $W$ is a standard $m$-dimensional Brownian motion. Assume further that
$f$ and $g$ are globally Lipschitz in $x$ with Lipschitz constants independent of $t$.

\begin{lem}\label{aaptoap}
Consider \eqref{SNE} and assume {\bf (H)}. If \eqref{SNE} admits an $L^2$-bounded solution $X$ on $\R$ which is
asymptotically almost periodic in distribution on $\R_+$, then \eqref{SNE} admits a solution $Y$ on $\R$ which is
almost periodic in distribution such that
\[
\lim_{t\to+\infty} \rho (\cal L(X(t)), \cal L(Y(t))) =0 \quad \hbox{and} \quad \|Y\|_\infty \le \|X\|_\infty.
\]
In particular, $\cal L(Y)$ is the almost periodic part of $\cal L(X)$.  The similar result holds when $X$ is asymptotically
almost periodic in distribution on $\R_-$.
\end{lem}

\begin{proof}
Denote $\mu(t):=\cal L(X(t))$ for each $t\in\R$ and fix a sequence $\alpha'\subset\mathbb{R}_+$ with $\alpha'_n\rightarrow\infty$. Since $f,g$ are uniformly almost periodic,
there is a subsequence $\alpha$ of $\alpha'$ such that $T_\alpha f$ and $T_\alpha g$ uniformly exist
on every $\mathbb R\times S$, with $S\subset \R^d$ being compact. By the proof of Theorem \ref{key}, the above subsequence $\alpha$ can be chosen
such that $T_\alpha\mu$ exists uniformly on any compact interval of $\R$.
On the other hand, since $\mu$ is asymptotically almost periodic on $\R_+$, there is a $\mathcal P(\mathbb R^d)$-valued almost periodic function $\eta$
such that $\lim_{t\to \infty} \rho(\mu(t),\eta(t))=0$. Note that the above subsequence $\alpha$
can be chosen such that $T_\alpha\eta$ uniformly exits on $\mathbb R$ by the almost periodicity of $\eta$
(so $T_\alpha\eta$ is almost periodic by Proposition \ref{uap-p2} (i)) and
\[
T_\alpha\mu(t)=T_\alpha \eta(t) \quad \hbox{for all } t\in\mathbb R.
\]
Also the proof of Theorem \ref{key} implies that $T_\alpha\mu$ (and hence $T_\alpha\eta$) is the law of some $L^2$-bounded solution $\tilde X(t)$,
with $\|\tilde X\|_\infty \le \|X\|_\infty$, of the limit equation
\[
\rmd X(t) = T_\alpha f(t,X(t)) \rmd t + T_\alpha g(t,X(t)) \rmd W(t).
\]


We take a subsequence of $\alpha$ if necessary (still denote it by $\alpha$) such that
\[
T_{-\alpha} T_\alpha f =f, ~T_{-\alpha} T_\alpha g =g
\]
uniformly on every $\mathbb R\times S$
and
\[
T_{-\alpha} T_\alpha \eta =\eta
\]
uniformly on $\mathbb R$. It follows from the proof of Theorem \ref{key} again that $T_{-\alpha} T_\alpha \eta$
is the law of some $L^2$-bounded solution $Y$ of the equation
\[
\rmd X(t) =T_{-\alpha} T_\alpha f(t,X(t)) \rmd t + T_{-\alpha} T_\alpha g(t,X(t)) \rmd W(t)
\]
with $\|Y\|_\infty \le \|\tilde X\|_\infty$.
That is, the almost periodic function $\eta$ is the law of the solution $Y$ of the equation \eqref{SNE}.

The proof in the case of $X$ being asymptotically almost periodic on $\R_-$ is similar.
\end{proof}


The following Amerio type result is Theorem B in the Introduction.

\begin{thm}\label{Amerio}
Consider \eqref{SNE}. Assume {\bf (H)} and that the Amerio positive (or negative)
semi-separation condition holds for \eqref{SNE} in $\cal D_r$. Then $|\cal D_r^{\eqref{SNE}}|$ is finite.
If $\cal B_r^{\eqref{SNE}}$ is non-empty, then it consists of almost periodic in distribution
solutions of \eqref{SNE}.
\end{thm}

\begin{proof}
We only consider the case of positive semi-separation since the negative semi-separation case is similar.

Firstly, $\cal D_{r}^{\eqref{SNE}}$ consists of finite number of elements, so the separation constant for
\eqref{SNE} depends only on the equation itself, independent of elements of $\cal D_{r}^{\eqref{SNE}}$.
Indeed, by Theorem \ref{key}, 
if there are infinite elements in
$\cal D_{r}^{\eqref{SNE}}$, then there exists a subsequence which converges, uniformly on
compact intervals, to an element of $\cal D_{r}^{\eqref{SNE}}$. But this limit cannot be positively semi-separated,
a contradiction to the Amerio semi-separation condition.

Secondly, the separation constant can be taken the same for all the hull equations of \eqref{SNE}. To see this, for any $(\tilde f,\tilde g)
\in H(f,g)$, by Proposition \ref{uap-p2} (iii) there exists a sequence $\alpha'$ with $\alpha'_n\to +\infty$ such that
$T_{\alpha'} (f,g)=(\tilde f,\tilde g)$ uniformly on $\R\times S$ for any compact $S\subset \R^d$.
Denote $\cal D_{r}^{\eqref{SNE}}=\{\mu_1,\ldots,\mu_\kappa\}$ and the separation constant for \eqref{SNE} by $d^{(f,g)}$.
Then there exists a subsequence $\alpha$ of $\alpha'$ such that
$T_\alpha\mu_i(\cdot), i=1,\ldots,\kappa$ exist uniformly on any compact interval and $T_\alpha\mu_i\in \cal D_{r}^{(\tilde f,\tilde g)}$
by the proof of Theorem \ref{key}. 
Note that $\inf_{t\in\R_+}\rho(\mu_i(t),\mu_j(t))\ge d^{(f,g)}$ implies that
\begin{equation}\label{sep}
\inf_{t\in\R_+}\rho(T_\alpha\mu_i(t),T_\alpha\mu_j(t))\ge d^{(f,g)}\qquad
\hbox{for } 1\le i,j\le \kappa \hbox{ and } i\neq j.
\end{equation}
That is, $T_\alpha\mu_i$ are distinct and so $|\cal D_{r}^{(\tilde f,\tilde g)}|\ge \kappa$.
Conversely, it follows from the fact $(f, g) \in H(\tilde f,\tilde g)$ that the above argument holds symmetrically. So
$|\cal D_{r}^{(f,g)}|$ is  no less than $|\cal D_{r}^{(\tilde f,\tilde g)}|$. This enforces that
$|\cal D_{r}^{(\tilde f,\tilde g)}|= |\cal D_{r}^{(f,g)}|$.
In the meantime, it follows from \eqref{sep} that $\cal D_{r}^{(\tilde f,\tilde g)}=\{T_\alpha\mu_1,\ldots,T_\alpha\mu_\kappa\}$
and $d^{(\tilde f,\tilde g)}\ge d^{(f,g)}$, with $d^{(\tilde f,\tilde g)}$ being the separation constant for the equation
$(\tilde f,\tilde g)$. By symmetry, we have $d^{(f,g)}\ge d^{(\tilde f,\tilde g)}$. So $d^{(f,g)}= d^{(\tilde f,\tilde g)}$, which
we denote by $d^{H(f,g)}$.

Thirdly, any $X\in\cal B_r^{\eqref{SNE}}$, with $\mu(\cdot)=\cal L(X(\cdot))$, is asymptotically almost periodic in distribution on $\R_+$.
Indeed, for any sequence $\eta' >0$ with $\eta'_n\rightarrow\infty$, there exists a subsequence
$\eta\subset\eta'$ such that $T_\eta (f,g)$ exists uniformly on $\R\times S$ for any compact $S\subset\R^d$
by the uniform almost periodicity of $f,g$ and $T_\eta \mu$ exists uniformly on any compact interval by
the proof of Theorem \ref{key}. Assume that sequences $\delta >0$, $\beta\subset\eta,\gamma\subset\eta$ are such that
\[
T_{\delta+\beta}\mu=\nu_1 \quad\hbox{and}\quad T_{\delta+\gamma}\mu=\nu_2
\]
exist pointwise on $\mathbb{R}_+$. By taking subsequences of $\delta,\beta,\gamma$ if necessary,  we may assume that
$T_{\delta+\beta}\mu$, $T_{\delta+\gamma}\mu$ exist uniformly on any compact interval of $\R$, and
\begin{align*}
T_{\delta+\beta} (f,g) = T_\delta T_\beta (f,g) = T_\delta T_\eta (f,g), \quad  T_{\delta+\gamma} (f,g) = T_\delta T_\gamma (f,g) = T_\delta T_\eta (f,g)
\end{align*}
uniformly on $\R\times S$ for any compact $S$. So by the proof of Theorem \ref{key} again, $\nu_1$ and $\nu_2$ are
the laws on $\R^d$ of solutions for the same equation $T_\delta T_\eta (f,g)$. Then it follows from the Amerio positive
separation condition that $\nu_1(t)\equiv \nu_2(t)$ or $\rho(\nu_1(t),\nu_2(t))\ge d^{H(f,g)}$ on $\R_+$. That is,
$\mu$ is an asymptotically almost periodic $\cal P(\R^d)$-valued mapping by Proposition \ref{aap-p}. Therefore, $X$ is
asymptotically almost periodic in distribution on $\R_+$.

Finally, the above given $X\in \cal B_r^{\eqref{SNE}}$ is actually almost periodic in distribution. To see this, note that
the almost periodic part $p$ of the law $\mu(\cdot)=\cal L(X(\cdot))$ is indeed the law of some solution of \eqref{SNE}
by Lemma \ref{aaptoap}.  That is, $\mu,p\in \cal D_r^{\eqref{SNE}}$. But $\lim_{t\to +\infty} \rho(\mu(t),p(t))=0$,
so the Amerio semi-separation condition enforces that $\mu(t)\equiv p(t)$ on $\R$. That is, $X$ is almost periodic in distribution.
The proof is complete.
\end{proof}

\begin{rem}
Note that, under the assumptions of Theorem \ref{Amerio}, we have $|\cal D_r^{(\tilde f,\tilde g)}|=|\cal D_r^{(f,g)}|$
for any hull equation $(\tilde f, \tilde g)$ of \eqref{SNE} and all the elements of $\cal D_r^{(\tilde f, \tilde g)}$ are almost
periodic.
\end{rem}

The following trivial separation case is very important in applications.

\begin{cor}\label{Amerio-unique}
Consider \eqref{SNE}. Assume {\bf (H)} and that each hull equation $(\tilde f, \tilde g)$ of \eqref{SNE}
admits a unique distribution in $\cal D_r$, i.e. $|\cal D_r^{(\tilde f, \tilde g)}| =1$. Then
$\cal B_r^{\eqref{SNE}}$ consists of almost periodic in distribution
solutions with the unique common distribution in $\cal D_r^{\eqref{SNE}}$.
\end{cor}

\begin{rem}
In the literature, the only applications of deterministic Amerio type theorems to specific models are the trivial separation case, i.e. there
is a unique solution in some given compact subset of $\R^d$ ($\cal D_r$ in our case); see, e.g. \cite{Fink74}. For stochastic equations,
it is certainly interesting to find (if possible) applications in nontrivial separation case.
\end{rem}

One weakness of Amerio type results (Theorem \ref{Amerio} and Corollary \ref{Amerio-unique}) is that they impose hypotheses on all hull equations.
This may be partly remedied by the inheritance property, as stated in the following result.

\begin{thm}\label{Amerio-inherit}
Consider \eqref{SNE} and assume {\bf (H)}. Assume that the property $P$ is negative semi-separating in $\cal D_r^{\eqref{SNE}}$ and inherited in distribution
in $\cal D_r$, and that the number of elements of $\cal D_r^{\eqref{SNE}}$ satisfying property $P$ is finite. Then every element of $\cal D_r^{\eqref{SNE}}$ with
property $P$ is almost periodic. In particular, \eqref{SNE} admits almost periodic in distribution solutions in $\cal B_r$.
\end{thm}

\begin{proof}
Assume that the elements of $\cal D_r^{\eqref{SNE}}$ with property $P$
are $\mu_1,\ldots,\mu_\kappa$, so the separation constant depends only on the equation \eqref{SNE}, which we denote by $d^{(f,g)}$.
That is, $\inf_{t\in\R_-}\rho(\mu_i(t),\mu_j(t))\ge d^{(f,g)}$ for $i,j=1,\ldots,\kappa$ and $i\neq j$.

We now show that any hull equation $(\tilde f,\tilde g)$ of \eqref{SNE} also admits $\kappa$ elements of
$\cal D_r^{(\tilde f,\tilde g)}$ with property $P$ and the separation constant for the equation $(\tilde f,\tilde g)$
can be chosen the same as that of  \eqref{SNE}. Indeed, by Proposition \ref{uap-p2} (iii) we may take a sequence $\alpha'$ with $\alpha'_n\to -\infty$ such that
$T_{\alpha'} (f,g)=(\tilde f,\tilde g)$ uniformly on $\R\times S$ for any compact $S\subset \R^d$.
Take a subsequence $\alpha$ of $\alpha'$ such that
$T_\alpha\mu_i(\cdot), i=1,\ldots,\kappa$ exist uniformly on any compact interval and $T_\alpha\mu_i\in \cal D_{r}^{(\tilde f,\tilde g)}$
by the proof of Theorem \ref{key}. Furthermore, $\inf_{t\in\R_-}\rho(T_\alpha\mu_i(t),T_\alpha\mu_j(t))\ge d^{(f,g)}$
for $i\neq j$ and it follows from the inheritance of property $P$ that
each $T_\alpha\mu_i$ also satisfies property $P$. That is, the equation $(\tilde f,\tilde g)$ admits at least $\kappa$ elements of
$\cal D_r^{(\tilde f,\tilde g)}$ with property $P$. By symmetry, there exists a sequence $\beta'$ with $\beta'_n\to-\infty$ such
that $T_{\beta'}(\tilde f,\tilde g)=(f,g)$, so the same argument yields that equation $(\tilde f,\tilde g)$ admits exactly
$\kappa$ elements of $\cal D_r^{(\tilde f,\tilde g)}$ with property $P$ which are $T_\alpha\mu_i, i=1,\ldots,\kappa$. The
symmetric argument also implies that the separation constant among these elements of $\cal D_r^{(\tilde f,\tilde g)}$ with property $P$
can be chosen the same as that of \eqref{SNE}, which we denote by $d^{H(f,g)}$.

The remaining proof is almost the same as that of Theorem \ref{Amerio}, just replacing $\R_+$ by $\R_-$. The proof is complete.
\end{proof}

In applications, the inheritance property is usually checked through stability properties. Since this topic itself deserves
a separate paper, here we will not discuss it further in this direction.

To conclude this section, we give a sufficient condition on the existence of $L^2$-bounded solutions.

\begin{thm}\label{L2exis}
Consider \eqref{SNE} and assume {\bf (H)}. If \eqref{SNE} admits a solution $\varphi$ on $[t_0, \infty)$ with
$\sup_{t\in[t_0, \infty)} \|\varphi(t)\|_2\le M$ for some $t_0\in\R$ and constant $M>0$, then \eqref{SNE} admits an $L^2$-bounded solution $\tilde\varphi$ on $\R$
satisfying $\|\tilde\varphi\|_\infty \le M$.
\end{thm}

\begin{proof}
By Proposition \ref{uap-p} (i)--(ii) and Proposition \ref{uap-p2} (iii), we may take a sequence $\alpha$ such that
$T_\alpha(f,g)=(f,g)$ uniformly on $\R\times S$ for any compact $S\subset \R^d$,
with $\alpha_n\to+\infty$ as $n\to\infty$.
Note that for any $t\ge s\ge t_0$
\[
\varphi(t) = \varphi(s) + \int_{s}^t f(r,\varphi(r))\rmd r + \int_{s}^t g(r,\varphi(r)) \rmd W(r)
\]
for some Brownian motion $W$. Similar to the proof of Lemma \ref{limit}, let $\varphi_n, f_n, g_n, W_n$ be defined the same as there.
Then $\varphi_n$ is defined on $[t_0-\alpha_n, \infty)$, satisfies $\sup_{r\in [t_0-\alpha_n, \infty)}\|\varphi_n(r)\|_2\le M$ and
\[
\varphi_n(t) = \varphi_n(s) + \int_s^t f_n(r,\varphi_n(r))\rmd r + \int_s^t g_n(r,\varphi_n(r)) \rmd W_n(r)
\]
for any $t\ge s \ge t_0-\alpha_n$.

For any fixed $a\in\R$, $\varphi_n$ is defined on $[a, \infty)$ when $n$ is large, and by Theorem \ref{key} there is
a subsequence of $\{\varphi_n\}$ which converges in distribution, uniformly on compact intervals of $[a, \infty)$,
to some $\tilde\varphi$ as $n\to\infty$ which satisfies
the limit equation (i.e. \eqref{SNE} itself) on $[a, \infty)$, i.e. for any $a\le s\le t$
\[
\tilde\varphi(t) = \tilde\varphi(s) + \int_s^t  f(r,\tilde\varphi(r))\rmd r + \int_s^t  g(r,\tilde\varphi(r)) \rmd \tilde W(r)
\]
for some Brownian motion $\tilde W$. By choosing $a$ to be a sequence converging to $-\infty$ and the standard diagonal method,
we may assume that the limit $\tilde\varphi$
satisfies \eqref{SNE} on $\R$. The Fatou's lemma implies that $\|\tilde \varphi\|_\infty \le \sup_{t\in[t_0, \infty)}\|\varphi(t)\|_2\le M$.
The proof is complete.
\end{proof}

\section{Applications}

In this section, we give some applications of our results.

\begin{thm}\label{Amerio-Lya}
Consider \eqref{SNE} and assume {\bf (H)}. Assume that the function $V:\R\times \R^d\to \R_+$ is $C^2$ in $t\in \R$ and $C^3$ in $x\in\R^d$,
and that the differentials $D^iV$ of $V$ with $i=0,1,2$ and the derivatives $V_{tx_ix_j}$, $V_{x_ix_jx_k}$, $i,j,k=1,\ldots, d$
are bounded on $\R\times S$ for any compact $S\subset\R^d$. Assume further that $V$ satisfies:
\begin{align}
\mathscr{L}V(t,x-y):= & \frac{\partial V}{\partial t}(t,x-y)+\sum_{i=1}^d(f_i(t,x)-f_i(t,y))\frac{\partial V}{\partial x_i}(t,x-y)\notag\\
&~ +\frac{1}{2}\sum_{i,j=1}^d \sum_{l=1}^m (g_{il} (t,x)-g_{il}(t,y))\frac{\partial^2V}{\partial x_i\partial x_j}(t,x-y) (g_{jl}(t,x) -g_{jl}(t,y)) \notag\\
\geq & a(|x-y|) \quad \text{ for all } t\in\R \text{ and } x,y\in\R^d, \tag{$h_0$}
\end{align}
\begin{equation}
\inf_{t\in \R}V(t,x)>0\text{ for each $x\neq 0$, and $V(t,0)=0$ for any $t\in \R$,} \tag{$h_1$}
\end{equation}
and
\begin{equation}
V(t,x)\leq b|x|^2+c\text{ for all $ (t,x)\in\R\times\mathbb{R}^d$ and some positive constants $b,c$,} \tag{$h_2$}
\end{equation}
where $a:\mathbb{R}_+\rightarrow\mathbb{R}_+$ is continuous, $a(0)=0$, $a(r)>0$ for $r>0$ and $\liminf_{r\rightarrow \infty}a(r)>0$.
Then $\cup_{r>0}\cal {D}_r^{\eqref{SNE}}$ is empty or consists of a unique element which is almost periodic in $t$.
\end{thm}

\begin{proof}
We divide the proof into 3 steps.

{\em Step 1. Uniqueness of strong $L^2$-bounded solutions.} Assume that ${X}$ and ${Y}$ are two strong ${L}^2$-bounded
solutions of \eqref{SNE} on $\R$ for given Brownian motion $W$.
Fix $t_0\in \R$. Then for any $t\geq t_0$,
\begin{align*}
X(t)-Y(t)=& X(t_0)-Y(t_0) + \int_{t_0}^t [ f(s,X(s))-f(s,Y(s))]\rmd s \\
&\quad + \int_{t_0}^t [ g(s,X(s))-g(s,Y(s))]\rmd W(s).
\end{align*}
We define a sequence of stopping times as follows:
\[
\tau_n :=\inf\{t\ge t_0: \max\{|X(t)|, |Y(t)|\} \ge n\}.
\]
Then we have by It\^o's formula and $(h_0)$
\begin{align*}
E V(t\wedge\tau_n, X(t\wedge\tau_n)-Y(t\wedge\tau_n)) & = E V(t_0,X(t_0)-Y(t_0)) + E \int_{t_0}^{t\wedge\tau_n}\mathscr{L}V(s,X(s)-Y(s))\rmd s \nonumber\\
& \ge E V(t_0,X(t_0)-Y(t_0))+ E \int_{t_0}^{t\wedge\tau_n} a (|X(s)-Y(s)|) \rmd s.
\end{align*}
Noting that $\tau_n\to \infty$ almost surely as $n\to \infty$, we have
\begin{equation}\label{LV}
E V(t, X(t)-Y(t))\ge E V(t_0,X(t_0)-Y(t_0))+ E \int_{t_0}^t a (|X(s)-Y(s)|)\rmd s.
\end{equation}
In particular,
\begin{equation}\label{V-mono}
{E}V(t,X(t)-Y(t))\geq E V(t_0,X(t_0)-Y(t_0)).
\end{equation}

By the ${L}^2$-boundedness of $X(t)$, $Y(t)$ and ($h_2$), the limit $\lim_{t\to \infty} {E}V(t,X(t)-Y(t))$ exists. This together with \eqref{LV}
implies that
\[
\lim_{n\to \infty} E \int_{n}^{\infty} a(|X(s)-Y(s)|)\rmd s =0.
\]
Since $a$ is nonnegative, there exists a sequence $\{t_n\}$, with $t_n\to \infty$ as $n\to \infty$,
such that $\lim_{n\to\infty}E a(|X(t_n)-Y(t_n)|) =0$. Since the function $a$ only vanishes at $0$ and $\liminf_{r\to \infty} a(r) >0$, this
enforces that $\lim_{n\to\infty} |X(t_n)-Y(t_n)| =0$ in probability.
So the Lebesgue dominated convergence theorem and $(h_1)$-$(h_2)$ yield that
\[
\lim_{n\to\infty} {E}V(t_n,X(t_n)-Y(t_n)) =0,
\]
which implies $E V(t_0,X(t_0)-Y(t_0))=0$ by \eqref{V-mono} and the non-negativeness of $V$. So $X(t_0)=Y(t_0)$ almost surely by $(h_1)$ again.
Since $f,g$ are global Lipschitz, we have $X(t)=Y(t)$ on $\R$ almost surely by Remark \ref{comm-key} (ii).

{\em Step 2. Convergence of $V$ and inheritance of $(h_0)$-$(h_2)$.}
For given sequence $\alpha'$, let $V_n(\cdot,\cdot):= V(\cdot+\alpha'_n,\cdot)$.
For any compact interval $I\subset\R$ and compact subset $S\subset\R^d$, since $V,V_t,V_{x_i}$ are bounded on
$\R\times S$, $V_n$ are uniformly bounded and equi-continuous on $I\times S$. So it follows from the Arzela-Ascoli theorem
that there exists a subsequence $\alpha$ of $\alpha'$ so that $T_\alpha V=\lim_{n\to\infty} V(\cdot+\alpha_n,\cdot)$
uniformly exists on $I\times S$; by the
diagonalization argument, the subsequence $\alpha$ may be chosen such that $T_\alpha V$ uniformly exists on any
compact subsets of $\R\times\R^d$. In the same way, by the hypothesis on the bounded differentials and derivatives of $V$,
the subsequence $\alpha$ can be further chosen such that $T_\alpha V_t$, $T_\alpha V_{x_i}$, $T_\alpha V_{x_ix_j}$
(the meaning of these notations is like $T_\alpha V$) uniformly exists on any compact subsets of $\R\times\R^d$.

Since $T_\alpha V$ and $T_\alpha V_t$ uniformly exist
on any compact subsets of $\R\times\R^d$, it follows that $\frac{\partial T_\alpha V}{\partial t}= T_\alpha V_t$
on $\R\times\R^d$. Similarly, we have
$\frac{\partial T_\alpha V}{\partial x_i} = T_\alpha V_{x_i}$ and  $\frac{\partial T_\alpha V_{x_i}}{\partial x_j} = T_\alpha V_{x_ix_j}$
on  $\R\times\R^d$ for $i,j=1,\ldots,d$, which implies further that $\frac{\partial^2 T_\alpha V}{\partial x_i \partial x_j}= T_\alpha V_{x_ix_j}$
on $\R\times\R^d$.

Since the above sequence $\alpha'$ is arbitrary, for given $(\tilde f, \tilde g)\in H(f,g)$, we
may assume that $\alpha'$ is such that $T_{\alpha'}(f,g)=(\tilde f, \tilde g)$ uniformly on $\R\times S$
for any compact $S\subset\R^d$. So, in this case we have
\[
T_\alpha (\mathscr LV) = \mathscr L_{T_\alpha(f,g)}(T_\alpha V) \quad \hbox{on } \R\times \R^d
\]
with
\begin{align*}
\mathscr L_{T_\alpha(f,g)} (T_\alpha V) & =  \frac{\partial T_\alpha V}{\partial t}(t,x-y)
+\sum_{i=1}^d(T_\alpha f_i(t,x)-T_\alpha f_i(t,y))\frac{\partial T_\alpha V}{\partial x_i}(t,x-y)\\
& +\frac{1}{2}\sum_{i,j=1}^d \sum_{l=1}^m (T_\alpha g_{il} (t,x)-T_\alpha g_{il}(t,y))
\frac{\partial^2T_\alpha V}{\partial x_i\partial x_j}(t,x-y) (T_\alpha g_{jl}(t,x) -T_\alpha g_{jl}(t,y)).
\end{align*}

It is immediate to see that $T_\alpha V$ satisfies $(h_1)$-$(h_2)$, and $\mathscr L_{T_\alpha(f,g)}(T_\alpha V)\ge a(|x-y|)$, i.e.
$(h_0)$ also holds with $T_\alpha V$, $T_\alpha f$, $T_\alpha g$ replacing $V,f,g$, respectively.

{\em Step 3. Conclusion. }  Consider the hull equation
\begin{equation*}
\rmd X = T_{\alpha}f(t,X)\rmd t+ T_{\alpha}g(t,X)\rmd W.
\end{equation*}
Clearly the hypothesis {\bf (H)} holds for this equation, so the same argument as in Step 1 implies that
this hull equation admits a unique strong $L^2$-bounded solution on $\R$ for given Brownian motion $W$.

By the proof of Theorem \ref{key}, we know that if the unique strong $L^2$-bounded solution $X$ of \eqref{SNE} satisfies
$\|X\|_\infty\le r_0$ for some $r_0>0$, then the unique $L^2$-bounded solutions of the hull equations are also bounded
with the same $r_0$. Note that the pathwise uniqueness implies the uniqueness of laws
on $\R^d$, so $|\cal D_{r_0}^{T_\alpha(f,g)}|=|\cal D_{r_0}^{(f,g)}|\le 1$ and hence $|\cup_{r>0}\cal D_r^{T_\alpha(f,g)}|=|\cup_{r>0}\cal D_r^{(f,g)}|\le 1$.
The result now follows from Corollary \ref{Amerio-unique}.
\end{proof}


Now we give some applications of Theorem \ref{Amerio-Lya}.

\begin{cor}
Consider one dimensional linear stochastic differential equation with $m$-dimensional Brownian motion:
\begin{equation}\label{ex2}
\rmd X(t)=(A(t)X(t)+f(t))\rmd t+ \sum_{i=1}^m (B_i(t)X(t)+g_i(t)) \rmd W_i(t),
\end{equation}
where $A$, $f$, $B_i$, $g_i$ are almost periodic functions. If there exists
some constant $c>0$ such that $2 A(t)+\sum_{i=1}^m B_i^2(t)\geq c$ for all $t\in \R$, then
all the ${L}^2$-bounded solutions of \eqref{ex2} share the same distribution on $\R$ which is almost periodic.
\end{cor}

\begin{proof}
Consider the nonnegative function $V(t,x)=x^2\exp\{\arctan t\}$ for $t\in\R$ and $x\in \R$.
Note that $(h_1)$-$(h_2)$ hold and
the derivatives of $V$ satisfy the boundedness condition on $\R\times S$ for any compact subset $S\subset\R$,
which are required in Theorem \ref{Amerio-Lya}.
Since
\begin{align*}
\mathscr{L}V(t,x-y) & =\frac{1}{t^2+1}V(t,x-y)+(2 A(t)+\sum_{i=1}^m B_i^2(t))V(t,x-y)\\
& \geq cV(t,x-y)\geq c \rme^{-\pi/2} (x-y)^2,
\end{align*}
the condition $(h_0)$ holds.  The result now follows from Theorem \ref{Amerio-Lya}.
\end{proof}

\begin{cor}
Consider one dimensional stochastic differential equation with $m$-dimensional Brownian motion:
\begin{equation}\label{ex1}
\rmd X(t)=f(t,X(t))\rmd t+\sum_{i=1}^m g_i(t)\rmd W_i(t),
\end{equation}
where $g_i$ are almost periodic, $f(t,x)$ is uniformly almost periodic and global Lipschitz in $x$.
If there is a constant $L_0>0$ such that
\begin{equation}\label{con-f}
L_0(x-y)^2\leq (f(t,x)-f(t,y))(x-y)\quad \text{ for all } t,x,y\in\R,
\end{equation}
then all the ${L}^2$-bounded solutions of \eqref{ex1} share the same distribution on $\R$ which is almost periodic.
\end{cor}

\begin{proof}
Consider the function $V:\mathbb{R}\times\mathbb{R}\rightarrow\mathbb{R}_+$ given by
\[
V(t,x)=\exp\{\arctan t\}\ln (x^2+1).
\]
It is immediate to check that $(h_1)$-$(h_2)$ hold and the derivatives of $V$ satisfy the boundedness condition on $\R\times S$ for any compact subset $S\subset\R$.
On the other hand,
\begin{equation*}
\begin{split}
\mathscr{L}V(t,x-y)= & \frac{1}{t^2+1}V(t,x-y)+\frac{2(x-y)(f(t,x)-f(t,y))\exp\{\arctan t\}}{(x-y)^2+1}\\
\geq & \frac{2L_0(x-y)^2\exp\{\arctan t\}}{(x-y)^2+1}.
\end{split}
\end{equation*}
That is, $(h_0)$ also holds. So the result follows from Theorem \ref{Amerio-Lya}.
\end{proof}

\section{Some discussions and improvements}

From the viewpoint of differential equations and dynamical systems, it is very natural to consider how
the distribution of solutions evolves with the time, as we did in previous sections. Almost periodicity
is an important recurrent property (in the sense of Birkhoff \cite{Bir}), which roughly means that the motion will
turn back repeatedly with any preassigned
small error. When an equation, with recurrent solutions (motions), is stochastically perturbed, does
the perturbed equation still admit recurrent motions in some sense? It is one of our main motivations to partly answer this problem;
and it seems that it is appropriate, by the results in previous sections, to consider the recurrent motions in
distribution sense on $\R^d$. However, some probabilists may prefer to consider properties they are interested in on path spaces,
i.e. they think that properties for sample functions are  more probabilistic. Due to this, in this section we discuss
the almost periodicity of solutions on the path space; we note that the similar concept was considered by Da Prato and Tudor \cite{DT}.

In Sections 3 and 4, we proved that under the Favard or Amerio separation condition (besides other conditions),
equations \eqref{SLE} and \eqref{SNE} admit solutions which are almost periodic in distribution
on $\R^d$. Indeed, under the Favard condition or the trivial Amerio separation condition, these
solutions are almost periodic in distributions on the path space $C(\R,\R^d)$. It is well-known that $C(\R,\R^d)$
is a separable complete metric space with the metric
\[
d(\omega_1,\omega_2)=\sum_{n=1}^\infty \frac{1}{2^n} \min\left\{1, \sup_{-n\le t\le n} |\omega_1(t)-\omega_2(t)|\right\},
\]
i.e. the convergence on the path space means the uniform convergence on any compact interval.

For any solution $X$ of \eqref{SNE}, it determines a distribution on $C(\R,\R^d)$.  Denote the shift mapping
\[
\hat\mu: \R\to \cal P(C(\R,\R^d)),~ t\mapsto \hat\mu(t):=\cal L(X(t+\cdot)),
\]
where $\cal P(C(\R,\R^d))$ stands for the space of probability measures on the path space and $\cal L(X(t+\cdot))$
means the law of the $C(\R,\R^d)$-valued random variable $X(t+\cdot)$. Note that $\cal P(C(\R,\R^d))$ is a
separable complete metric space (see, e.g. \cite[Chapter II, Theorems 6.2 and 6.5]{P}). The solution $X$ is said to be
{\em almost periodic in strong distribution sense} if $\hat \mu$ is a $\cal P(C(\R,\R^d))$-valued almost periodic mapping.
It is clear that if $X$ is almost periodic in strong distribution sense, then it is almost periodic in distribution.

Firstly the result of Theorem \ref{key} can be improved.

\medskip
\noindent {\bf Theorem \ref{key}'.}
{\em Consider the following family of It\^o stochastic equations on $\R^d$
\[
\rmd X = f_n(t,X)\rmd t + g_n(t,X)\rmd W, \quad n=1,2,\cdots,
\]
where $f_n$ are $\R^d$-valued, $g_n$ are $(d\times m)$-matrix-valued, and
$W$ is a standard $m$-dimensional Brownian motion.
Assume that $f_n, g_n$ satisfy the conditions of global Lipschitz and linear growth with common Lipschitz and linear growth constants;
that is, there are constants $L$ and $K$, independent of $t\in\mathbb R$ and $n\in\N$, such that for all $x,y\in\mathbb R^d$
\begin{align*}
&|f_n(t,x)-f_n(t,y)|\vee |g_n(t,x)-g_n(t,y)|\le L|x-y|,\\
& |f_n(t,x)|\vee |g_n(t,x)|\le K (1+|x|).
\end{align*}
Assume further that $f_n\to f$, $g_n\to g$ pointwise on $\mathbb R\times \mathbb R^d$ as $n\to\infty$ and that
$X_n\in \cal B_{r_0}^{(f_n,g_n)}$ for some constant $r_0$, independent of $n$. Then there is a subsequence of $\{X_n\}$ which
converges, in strong distribution sense (the meaning is obvious), to some $X\in \cal B_{r_0}^{(f,g)}$.}

\begin{proof}
We only need to point out the difference from the proof of Theorem \ref{key}.
Note that, in the proof of Theorem \ref{key}, \eqref{appr} actually
implies that $\tilde X_n$ converges in distribution to $\tilde X$ on $C([a,b],\R^d)$.
Since the interval $[a,b]$ is arbitrary, the convergence indeed occurs on $C(\R,\R^d)$ in distribution.
\end{proof}

By Theorem \ref{key}' and minor revising the proof of Lemma \ref{uniap}, we have the following result.

\medskip
\noindent {\bf Lemma \ref{uniap}'.}
{\em Assume that each hull equation of \eqref{SNE} admits a unique minimal solution in strong distribution sense, i.e. all the minimal
solutions of the given hull equation possess the same law on the path space. Then these minimal solutions are almost
periodic in strong distribution sense.}
\medskip

Therefore, we have the following

\medskip
\noindent {\bf Theorem A'.} {\em
Consider \eqref{SLE} with the coefficients $A, B_1,\ldots, B_m$, $f,g_1,\ldots,g_m$ being almost periodic.
Assume further that \eqref{SLE} admits an $L^2$-bounded solution $X$, and that the
Favard condition holds for \eqref{SLE}. Then \eqref{SLE} admits a solution which is almost periodic in strong distribution sense.}

\begin{proof}
The theorem follows immediately from Lemmas \ref{exis}, \ref{Fav}, Remark \ref{rem-Fav}, Corollary \ref{mini}, and Lemma \ref{uniap}'.
\end{proof}

\begin{rem}
Note that Remark \ref{BN} and Corollary \ref{thAco} can be correspondingly improved by Theorem A'.
\end{rem}

For the nonlinear equation \eqref{SNE}, under the trivial Amerio separation condition, we have the following

\medskip
\noindent {\bf Corollary \ref{Amerio-unique}'.}
{\em Consider \eqref{SNE}. Assume {\bf (H)} and that each hull equation $(\tilde f, \tilde g)$ of \eqref{SNE}
admits, in $\cal B_r^{(\tilde f,\tilde g)}$, a unique distribution in strong sense, i.e.
all the elements of $\cal B_r^{(\tilde f,\tilde g)}$ share the same distribution on the path space. Then
$\cal B_r^{\eqref{SNE}}$ consists of solutions which are almost periodic in strong distribution sense
with the unique common distribution on the path space.}

\begin{proof}
The proof is completely similar to that of Lemma \ref{uniap}' since only the uniqueness of distribution on the
path space is essential in the proof.
\end{proof}

\begin{rem}
We note that the conclusion of Theorem \ref{Amerio-Lya} (and hence its applications) can be correspondingly
improved by Corollary \ref{Amerio-unique}'.
\end{rem}

\section*{Acknowledgements}
We sincerely thank Professor Yong Li for his kind pointing out to us a mistake in an early version of the proof of
Theorem \ref{key}, and thank Professor Zhao Dong for his kind suggesting that we give sufficient conditions on the
existence of $L^2$-bounded solutions.

\end{document}